\documentclass{article}

\usepackage[utf8]{inputenc}
\usepackage[margin=1in]{geometry}

\usepackage{amsmath}               
\usepackage{amsfonts,amssymb}              
\usepackage{amsthm}                
\usepackage{color}
\usepackage{enumerate}

\usepackage{bbm}
\usepackage{graphicx}

\newtheorem{theorem}{Theorem}
\newtheorem{example}{Example}
\newtheorem{remark}{Remark}

\newcommand{\R}{\mathbb R}

\newcommand{\tod}{\stackrel{d}{\longrightarrow}}

\newcommand{\E}{\mathbf {E}}

\newcommand{\pr}{\mathbb{P}}

\newcommand{\ind}{\mathbbm{1}}
\newcommand{\eqd}{\stackrel{d}{=}}
\newcommand{\N}{\mathbb N}
\newcommand{\B}{{\cal B}_{\R^d}}

\title{A Fubini--type limit theorem for the integrated  hyperuniform infinitely divisible moving averages}
\author{Evgeny Spodarev\thanks{Institute of Stochastics,
Ulm University, Germany; evgeny.spodarev@uni-ulm.de}}
        \date{\today}

\begin{document}

\maketitle

\begin{abstract}
    This short note shows a limiting behavior of  integrals of some centered antipersistent stationary infinitely divisible moving averages as the compact integration domain in $d\ge 1$ dimensions extends to the whole positive quadrant $\R^d_+$. Namely, the weak limit of their finite dimensional distributions is again a moving average with the same infinitely divisible purely jump integrator measure (i.e., possessing no Gaussian component), but with an integrated kernel function. The results apply equally to time series ($d=1$) as well as to random  fields ($d>1$). Apart from the existence of the expectation, no moment assumptions on the moving average are imposed allowing it to have an infinite variance as e.g. in the case of $\alpha$--stable moving averages with $\alpha\in(1,2)$ . If the field is additionally square integrable, its covariance integrates to zero (hyperuniformity).
\end{abstract}

{\bf Keywords}: {Stationary random field, moving average, infinitely divisible, ambit random field, integral spectral representation, $\alpha$--stable, time series, heavy tails, hyperuniformity, antipersistence}.
 \\
 
{\bf AMS subject classification 2020}: {Primary 60F05; Secondary 60G60, 60E07, 60E10} \\

\section{Introduction} \label{sectIntro}

The major object of interest of this note is   a real-valued stationary measurable infinitely divisible (ID, for short) moving average random field $X=\{X(t),t\in \R^d\}$ on a $d$--dimensional Euclidean index space defined by  its spectral representation
\begin{equation}\label{eqX0}
X(t) = \int_{\R^d} f(t-x) \ \Lambda(d x), \quad t \in \mathbb{R}^d, \; d\ge 1,
\end{equation}
 where $\Lambda$ is a 
 centered stationary ID independently scattered purely jump random measure and $f: \R^d\to \R$ is a suitable $\Lambda$--integrable kernel. 
Such random functions  are often used  as mathematical models in many application contexts varying from turbulent liquid flows \cite{barndorff2007ambit} (where they form a special case of {\it ambit random fields}, see \cite{Podol15,MR3839270} for a review) to radial growth of trees and tumors \cite{Jonsdottiretal05,Jonsdottiretal08}. An important special case of \eqref{eqX0} constitute continuous--time ($d=1$) stationary moving averages  governed by L\'evy processes (e.g. the so--called CARMA \cite{BrockLind24}) which are widely used in mathematical insurance and finance \cite{BarndorffScheph01,BrockLind12,Benthetal14,MR3397093,MR3451179}. They often do not have a finite variance describing high volatility of financial risks. Whereas the limiting behaviour of some  classes of functionals (both linear and non--linear) of random fields  \eqref{eqX0} is relatively well understood if the integrator $\Lambda$ is Gaussian (see e.g. \cite{MR4546115,MakSpo22} and references therein) or has at least finite second or higher moments  \cite{MR0885142,MR2046777,BulSpoTim12,MR4430002}, 
the case of a heavy--tailed $\Lambda$ is far less investigated. Few central limit theorems are available \cite {MR1733161,MR2047688,Karcher_thesis,Spod14,BPTh20,MR4367889} in the case of short range dependent $X$ as well as some results of non--central type \cite{MR3737916,MR4254797,MR1919616} for $X$ with long memory. The novel contribution of this note is a weak limit of finite dimensional distributions (FDDs, for short) of the sample integral of $X$ over a cubic infinitely growing sliding domain when we only assume  the existence of  the mean $\E X(0)=0$. The kernel function $f$ has to be integrable over $\R^d$ to zero  and factorize as
$$
f(x)=\prod_{j=1}^dg_j^\prime (x_j), \quad x=(x_1,\ldots, x_d),  
$$
where  differentiable functions $g_j:\R\to\R$ belong to the Sobolev space $W^{1,1}(\R)$ a.e.,  $j=1,\ldots, d$. 
In case of square integrable $X$, the above form of $f$ implies that the covariance function of $X$ also integrates to zero. This phenomenon (widely known in limit theorems for random functions as an intermediate regime between short and long  memory) is called  {\it antipersistence} \cite{MR3075595}. In physics, it is termed {\it hyperuniformity} and plays an important role in considering the geometric structure of matter (such as e.g. quasicrystals) which exhibits some regularity while being essentially random \cite{PhysRevE.94.022122,PhysRevE.94.022122,10.1063/1.4989492}. Recent progress  in search for hyperuniformity  has been made in  experimental physics,  materials science and biology \cite{TORQUATO20181,PhysRevX.11.021028}.

We show that all FDDs  of sample averages
$$\int_{ [0,T]^d+l} X(t) \, dt, \qquad l\in\R^d, \quad T>0,$$
converge as $T\to+\infty$ to the FDDs of the random field
$$
(-1)^d  \int_{\R^d}  \prod_{j=1}^d g_j(l_j-x_j) \, \Lambda(d x_1, \ldots, d x_d) ,\quad l=(l_1,\ldots, l_d)\in\R^d.
$$
For simplicity, this result is first proven for the case of stochastic time series ($d=1$, See Theorem \ref{thm:LTMA}) and is then generalized to all $d>1$ in  Theorem  \ref{thm:LTMAd}.
Interestingly enough, the limiting process belongs to the same family of ID moving averages with integrator $\Lambda$, but with a different kernel. This  Fubini--type result for
sample path integrals of $X$ (valid due to the linearity of integration operator combined with the convergence of the corresponding characteristic functions) can be used to understand the behaviour of spatial averages of hyperuniform structures. To some extent, it can be seen as a generalization of the Fubini theorem \cite[Theorem 4.1]{Samorodnitsky92} for the sample path integrals of $\alpha$--stable processes. For symmetric ID processes, a result with a similar Fubini flavor can be found in \cite[Theorem 5.1]{BraSam98}.

The paper is organized as follows: in Section \ref{sectNote},  common notation used throughout the note is given. Section \ref{sectPrelim} recalls some basic facts from the theory of infinitely divisible random functions with integral spectral representation needed in the sequel. In Section \ref{sectMain} our main results are formulated and proved in Theorems \ref{thm:LTMA} and \ref{thm:LTMAd}. Remark \ref{rem:anipers} explains why the considered random fields are  antipersistent and hyperuniform. The main results are illustrated by some examples.

\section{Notation}\label{sectNote}

Denote by $\B$ the $\sigma$-algebra of Borel subsets of $\R^d$, $d\ge 1$. Let $| \cdot |_d$ be the  Lebesgue measure (volume) in $\R^d$. In the sequel, the expression {\it almost everywhere} (a.e.) will be understood with respect to $| \cdot |_d$. Let $L^p(\R^d)$ be the set of all real-valued functions that are absolutely integrable on $\R^d$ with their power $p>0$. 
We use the standard notation $ W^{1,1}(\R)$ for the Sobolev space $\{ g\in L^1(\R): g^{\prime} \in L^1(\R)  \} $, where $g^{\prime} $ is the derivative of $g$. Let $\mbox{\rm sgn}(x) $ be the sign function of a real argument  $x$. Denote by $\ind(A)$ the indicator function of a set $A$. 

\section{Preliminaries on ID moving averages} \label{sectPrelim}

Let $X=\{X(t),t\in \R^d\}$ be  a real-valued measurable ID moving average random field defined on a probability space $(\Omega,{\cal F}, \pr)$ by the integral representation
\begin{equation}\label{eqX}
X(t) = \int_{\R^d} f(t-x) \ \Lambda(d x), \quad t \in \mathbb{R}^d, 
\end{equation}
 where $\Lambda$ is a 
 centered stationary ID independently scattered purely jump random measure with Lebesgue control measure. The absence of Gaussian part means that the characteristic function  $\varphi_{\Lambda(B)}$ for bounded Borel subsets  $B$ of $\R^d$ has the L\'evy-Khintchin representation
\begin{equation}\label{eqCharFLambda}
\log  \varphi_{\Lambda(B)}(z)= |B|_d \int_\R \left(  e^{izy}-1-izy\chi(y)   \right) \, \nu(dy), \quad z\in\R,  
\end{equation}
 where 
  $
 \chi(y) =\ind(|y|\le 1)$, $y\in \R$,
 and
 $\nu(\cdot)$ is the L\'evy measure on $\R$: $ \int_\R \min\{ y^2,1\} \nu(dy)<\infty$.
The kernel function $f$ above is $\Lambda$--integrable, i.e. satisfies the following two conditions (cf. \cite[Theorem 2.7, p. 461]{RajRos89}):

\begin{enumerate}
\item[(i)]  $ \int\limits_{\R^d} |f(s)| \left|  \int\limits_{\{1\le |x|\le |f (s)|^{-1} \}}  x  \,\nu(dx)  - \int\limits_{ \{|f (s)|^{-1}\le |x|\le 1 \} } x  \,\nu(dx) \right| \, d s<\infty,$
\item[(ii)] $ \int\limits_{\R^d}  \int\limits_{\{ |x|\ge |f (s)|^{-1} \} }  \nu(dx)  \, ds<\infty,$ 
\item[(iii)] $\int\limits_{\R^d} |f(s)|^2  \int\limits_{\{ |x| <|f (s)|^{-1} \} }  x^2  \,\nu(dx)\, ds  <\infty.$
\end{enumerate}

\begin{example}[\cite{BPTh20}, p.3]
Let  $\Lambda$ be symmetric with L\'evy measure $\nu(\cdot)$ having a density 
 $$
 \psi(x)\le \frac{C}{|x|^{1+\beta}}, \quad x\in \R\setminus\{0 \}, \quad \beta\in(0,2),
 $$
 where $C>0$ is a constant. The conditions (i) --(iii) are satisfied whenever $f\in L^\beta (\R^d)$, e.g. for
 $$
 |f(x)|\le K\left(  |x|^\gamma \ind(|x|<1)+  |x|^{-\alpha} \ind(|x|\ge 1) \right), \quad x\in\R^d, \quad K,\alpha>0,\quad \gamma\in\R,
 $$
 with $\alpha\beta>d$, $\gamma\beta> - d$.
 \end{example}
 
 By \cite[Proposition 2.10]{Ros18}, it holds
 \begin{equation}\label{eq:IsomorphRos}
 X(t)\eqd \int_{\R^{d+1}} y f(t-s) \left[  \eta(ds,dy)- \chi(   y f(t-s) )   ds \, \nu(dy)   \right]  +b , \quad t \in \mathbb{R}^d, 
  \end{equation}
 where $\eta(\cdot,\cdot)$ is a Poisson point process on $\R^d\times \R$ with intensity measure $ ds \, \nu(dy) $, and 
 $$
 b= \int_{\R^{d+1}} y f(s) \left[ \chi(   y f(s) ) - \chi(   y  )   \right]  ds \, \nu(dy)   .
 $$
 The log characteristic function of $X(t)$ writes (\cite[p. 462]{RajRos89})
\begin{equation*}
  \begin{split}
 \log \E e^{izX(t)} & =iz \int_{\R^{d+1}} y f(t-s) \left[ \chi(   y f(t-s) ) - \chi(   y  )   \right]  ds \, \nu(dy) \\
& + \int_{\R^{d+1}}  \left[   e^{izy f(t-s)}-1- iz y f(t-s) \chi(   y f(t-s) )  \right]  ds \, \nu(dy) , \quad z\in\R.
 \end{split}
\end{equation*}
 Now assume 
  \begin{equation}\label{eqf}
f\in L^1 (\R^d),
\end{equation} 
 \begin{equation}\label{eqNu}
 \int_{\R} |y|  \, \nu(dy) <\infty,
\end{equation} 
which evidently imply that conditions (i), (ii) hold true and  also
yield the integrability of $X(t)$.
 Under these assumptions, the above characteristic function simplifies to 
 \begin{equation}\label{eqCharfX(t)}
 \log \E e^{izX(t)}  =-iz a
 + \int_{\R^{d+1}}  \left[   e^{izy f(t-s)}-1 \right]  ds \, \nu(dy) , \quad z\in\R,
\end{equation}
 where
 $$
 a=\int_{\R^{d}}  f(s) \, ds \int_{-1}^1  y \, \nu(dy) .
 $$
 Interpreting \eqref{eqCharfX(t)}, we may write
 \begin{equation}\label{eqX(t)SN}
 X(t)\eqd \int_{\R^{d+1}} y f(t-s) \eta(ds,dy)- a,  \quad t \in \mathbb{R}^d
\end{equation}
 as a stationary shot noise random field.
 If we (additionally to conditions \eqref{eqf}, \eqref{eqNu}) require 
 \begin{equation}\label{eqNu2}
f\in L^2 (\R^d),\quad \int_{\R} y^2  \, \nu(dy) <\infty
\end{equation}
the field $X$ becomes square integrable with covariance function $C(t)=\E [ X(0) X(t) ]$ equal to
 \begin{equation}\label{eqCov}
C(t)= \int_{\R} y^2  \, \nu(dy) \int_{\R^d} f(-x) f(t-x)   \, dx, \quad t\in \R^d,
\end{equation}
compare \cite[Relation 2.44]{Karcher_thesis}.
\section{Main results} \label{sectMain}

First, consider the case $d=1$ of stationary time series moving averages \eqref{eqX}. Introduce the sample path integral of $X$ by
$$
S_{T,l}:=\int_l^{T+l} X(t) \, dt, \qquad l\in\R, \quad T>0.
$$
If $\E|X(0)|<\infty $, $S_{T,l}$ exists a.s.  by measurability of $X$ and and is integrable by stationarity and  Fubini's theorem: 
$$\E|S_{T,l}|\le  \int_l^{T+l} \E|X(t)| \, dt=T \E|X(0)|  < \infty. $$

\begin{theorem} \label{thm:LTMA}
Let $X=\{X(t),t\in\R\}$ be the stationary measurable ID moving average with integral representation \eqref{eqX} and kernel function $f(x)=g^\prime (x)$ a.e. on $\R$, where
$g\in W^{1,1}(\R)$ a.e. Assuming that conditions \eqref{eqNu} and (iii) hold, one gets the weak convergence of all FDDs of
$\{ S_{T,l},\; l\in\R\}$ as $T\to + \infty$ to those of   $\{ -  \int_\R g(l-s) \Lambda(ds) ,\quad l\in\R\}$.
\end{theorem}

\begin{remark}\label{rem:anipers}
\begin{enumerate}
\item
Since  $f(x)=g^\prime (x)$ a.e. on $\R$ and $g\in L^1(\R)$, it holds $\int_\R f(x)\,dx=\int_\R dg(x)=0$ and thus $a=0$ in representation \eqref{eqX(t)SN}. This yields that $X$ is centered: $\E X(t)=0$ . 
\item As the kernel function $f$  reflects the dependence structure of the process $X$, the case $\int_\R f(x)\,dx=0$ may be called {\it antipersistence}, similar to the case $\int_\R C(t)\,dt=0$ for stationary square integrable moving averages $X$ with covariance function $C$. Indeed, assuming the conditions \eqref{eqf}, \eqref{eqNu}, \eqref{eqNu2} the relation \eqref{eqCov} for $C(t)$  yields
$$
\int_{\R} C(t) \, dt= \int_{\R} y^2  \, \nu(dy) \int_{\R} f(-x) \int_{\R}  f(t-x)\, dt  \, dx=  \int_{\R} y^2  \, \nu(dy) \left( \int_{\R} f(x) \, dx \right)^2  =0
$$
by Fubini theorem.
Hence, the condition $\int_\R f(x)\,dx=0$  indicates a phenomenon called in physics {\it hyperuniformity} which is characterized by very regular shapes of sample paths of  a time series $X$. 
\end{enumerate}
\end{remark}
\begin{proof}[Proof of Theorem \ref{thm:LTMA}]
For any $m\in\N,$ $z_1, \ldots, z_m \in  \R$, write
\begin{equation*}\label{eqCharfX(t)Mult}
 \log \E e^{i \sum_{j=1}^m z_j S_{T,l_j}}   =  \log \E e^{i  \int_\R \sum_{j=1}^m z_j \ind( t\in [l_j, T+l_j]) X(t) \, dt } 
 = \int_{\R^2} \left( e^{iy J_T(s)}-1   \right)ds\, \nu(dy),
\end{equation*}
  where
 \begin{equation*}
 \begin{split}
 J_T(s):=& \int_\R \sum_{j=1}^m z_j \ind( t\in [l_j, T+l_j]) f(t-s) \, dt=\sum_{j=1}^m z_j \int_{l_j-s}^{T+l_j-s} f(x) \, dx \\
  & \longrightarrow \sum_{j=1}^m z_j \int_{l_j-s}^{+\infty} f(x) \, dx =- \sum_{j=1}^m z_j g(l_j-s) \quad \mbox{as } T\to +\infty.
  \end{split}
\end{equation*}
 Hence, 
 \begin{equation*} 
 \log \E e^{i \sum_{j=1}^m z_j S_{T,l_j}}    \longrightarrow
  \int_{\R^2} \left( e^{-iy \sum_{j=1}^m z_j g(l_j-s)}-1   \right)ds\, \nu(dy),\quad T\to +\infty,
\end{equation*}
 and finally
 \begin{equation*} 
\left( S_{T,l_1}, \ldots,  S_{T,l_m} \right)    \tod   - \left(  \int_{\R^2} y g(l_1-s) \, \eta(ds,dy)    , \ldots,  \int_{\R^2} y g(l_m-s) \, \eta(ds,dy)   \right) 
 \quad \mbox{as } T\to +\infty.
\end{equation*}
 This together with \eqref{eqX(t)SN} finishes the proof.
 
 \end{proof}

\begin{example}
\begin{enumerate}
\item Functions $g(x)=e^{-|x|}$, $g(x)=e^{- P_k(x)}$, $x\in\R$, belong to $C^1(\R) \cap W^{1,1}(\R)$ a.e. where $P_k(x)$ is a polynomial of even degree $k\ge 2$ with
a positive coefficient at $x^k$. The corresponding ID time series are $$X(t)=   \int_\R  \mbox{\rm sgn}(t-x) e^{-|t-x|} \, \Lambda(dx)=e^{-t} \int_{-\infty}^t e^x \, \Lambda(dx) -e^{t} \int^{+\infty}_t e^{-x} \,  \Lambda(dx),   \quad   t\in\R ,$$  which we may call (in analogy to Gaussian integrators $\Lambda$) a {\it signed two--sided ID Ornstein--Uhlenbeck process}, and $$X(t)= \int_\R P_k^\prime(t-x)e ^{- P_k(t-x)}  \,  \Lambda(dx) , \quad   t\in\R.$$
\item An important integrator measure $\Lambda$ is the so--called Dickman subordinator with L\'evy measure
$\nu(dy)=\frac1y \ind\left(y\in (0,1 )\right) dy$. For this $\nu$, it can be easily seen that conditions 
\eqref{eqNu}, \eqref{eqNu2} hold true. Taking a kernel function $f\in L^1(\R)\cap L^2(\R) $ yields a square integrable time series \eqref{eqX}.  Additionally assuming $\int_{-\infty}^\cdot f(x)\, dx \in L^1(\R)$  makes Theorem \ref{thm:LTMA} hold true for this $X$. 
See references \cite{Caravennaetal19,Guptaetal24} for more information on the main properties and use of Dickman subordinators.  
\item Examples of non--$\alpha$--stable integrator measures $\Lambda$ with finite mean, but infinite variance can be given e.g.  by Inverse Gamma subordinator $IG(\alpha,\beta)$
and Student $t_{\tau}$--distributed random measures with a specific choice of form parameters $\alpha, \tau \in(1,2]$ and an arbitrary $\beta>0$. Their L\'evy measures are given by 

$$ \nu(y)= \frac{1}{\pi^2} \int_0^\infty  \frac{e^{-sy}\, ds}{s\left( J^2_\alpha(2\sqrt{\beta s}) + Y^2_\alpha(2\sqrt{\beta s})  \right) } \,  \frac{dy}{ y},\quad y>0, $$ and 

$$ \nu(y)=  \frac{2^{3/4}}{\pi^{5/2} \sqrt{|y|}} \int_0^\infty  \frac{K_{1/2} \left(\sqrt{2 s} |y|\right)      \, ds  } { s^{3/4} \left( J^2_{\tau/2} (\sqrt{2\tau s} ) + Y^2_{\tau/2}(\sqrt{2\tau s}) \right) } \,  {dy},\quad y\in \R,$$ respectively (cf. \cite{BarndorffScheph01,Grig13,Massing18}), where $J_\alpha,$ $Y_\alpha,$ are Bessel functions of order $\alpha$ of the first  and second kind, and $K_\alpha$ is a modified Bessel function of order $\alpha$ of the second kind.

\item Set $m=1$, $l_1=0$ in the proof of Theorem \ref{thm:LTMA}: we have
$$
S_T:=S_{T,0}=\int_0^T X(t)\, dt \tod  \int_{\R^2} y g(s) \, \eta(ds,dy)  \quad \mbox{as } T\to +\infty,
$$
since $ - \int_{\R^2} y g(-s) \, \eta(ds,dy) \eqd  \int_{\R^2} y g(s) \, \eta(ds,dy)$ by substitution $s\mapsto -s$.
\end{enumerate}
\end{example}
Theorem \ref{thm:LTMA} can be easily generalized to stationary ID random fields $X$ with spectral representation \eqref{eqX}  and  the dimension of index space $d>1$. For a sample integral
$$
S_{T,l}:=\int_{ [0,T]^d+l} X(t) \, dt, \qquad l\in\R^d, \quad T>0,
$$
we formulate the following
\begin{theorem} \label{thm:LTMAd}
Let $X=\{X(t),t\in\R^d\}$ be a  stationary measurable  ID random field with  spectral representation \eqref{eqX} and kernel function 
$f(x)= \prod_{j=1}^d f_j(x_j) $, $x=(x_1,\ldots, x_d), $   $f_j(x_j)= g_j^\prime (x_j)$ a.e. on $\R$, where
$g_j\in W^{1,1}(\R)$ a.e.,  $j=1,\ldots, d$.  Under conditions \eqref{eqNu} and (iii)  it holds the weak convergence of all FDDs of
$\{ S_{T,l},\; l\in\R^d\}$ as $T\to + \infty$ to those of  $\{ Y_l,\; l\in\R^d\}$, where
$$ Y_l=  (-1)^d  \int_{\R^d}  \prod_{j=1}^d g_j(l_j-x_j) \, \Lambda(d x_1, \ldots, d x_d) ,\quad l=(l_1,\ldots, l_d)\in\R^d.$$
\end{theorem}
The proof of Theorem \ref{thm:LTMAd} is analogous to the above case $d=1$ of Theorem \ref{thm:LTMA} and is thus omitted. \\

Remark \ref{rem:anipers} evidently stays valid for the ID random fields considered in Theorem \ref{thm:LTMAd} since the
$d$--dimensional integration for kernels $f(x)= \prod_{j=1}^d f_j(x_j) $ splits into $d$ univariate integrals with respect to variables $x_j$, $j=1,\ldots, d$. Hence, stationary ID moving averages 
$$X(t)=  \int_{\R^d}  \prod_{j=1}^d g^\prime_j(t_j-x_j) \, \Lambda(d x_1, \ldots, d x_d) ,\quad t=(t_1,\ldots, t_d)\in\R^d
$$
for $g_j\in C^1(\R) \cap W^{1,1}(\R)$ a.e.,  $j=1,\ldots, d$
are antipersistent and hyperuniform. 

\section*{Acknowledgement}\label{sect:Ackno}

Evgeny Spodarev gratefully acknowledges the partial financial support of his research visit to Cardiff University in November 2021 provided by the London Mathematical Society (Grant No. 42007) and German Research Society (Grant No. 39087941). He also thanks  Michael Klatt, Nikolai Leonenko and Alexander Lindner  for fruitful discussions and some literature references.

\newcommand{\noopsort}[1]{} \newcommand{\printfirst}[2]{#1}
  \newcommand{\singleletter}[1]{#1} \newcommand{\switchargs}[2]{#2#1}

\end{document}